\documentclass[12pt]{article}
\usepackage{times}
\ifdefined\usebigfont

\usepackage[fontsize=13pt]{scrextend}
\usepackage[left=1.56in,right=1.56in,top=1.74in,bottom=1.74in]{geometry}
\usepackage{natbib}
\else
\fi

\RequirePackage{amsthm,amsmath,amsmath,amsfonts,amssymb}

\makeatletter
\newtheorem*{rep@theorem}{\rep@title}
\newcommand{\newreptheorem}[2]{%
	\newenvironment{rep#1}[1]{%
		\def\rep@title{#2 \ref{##1}}%
		\begin{rep@theorem}}%
		{\end{rep@theorem}}}
\makeatother

\usepackage{mathabx}

\usepackage{algorithmic}
\usepackage[ruled,noend]{algorithm2e}
\usepackage[skins]{tcolorbox}
\newcommand{\sign}{\mathrm{sign}}

\usepackage{color}

\usepackage{bm}
\usepackage{hyperref}
\hypersetup{colorlinks,
	linkcolor=blue,
	citecolor=blue,
	urlcolor=magenta,
	linktocpage,
	plainpages=false}

\usepackage{enumerate}

\usepackage{natbib}
\usepackage{amsmath}
\usepackage{amsthm}
\usepackage{amssymb}
\usepackage{mathabx}
\usepackage{tikz}
\usepackage{xcolor}
\usetikzlibrary{arrows}

\usepackage[title]{appendix}

\usepackage{hyperref}
\usepackage{bm}

\allowdisplaybreaks[4]

\usepackage{caption}
\usepackage{subcaption}

\newcommand{\conv}{\mathrm{conv}}

\newcommand{\reals}{\mathbb{R}}

\def\R{\mathbb{R}}

\newcommand{\norm}[1]{\left\|#1\right\|}

\newcommand{\relu}[1]{\textrm{ReLU}\left(#1\right)}

\newtheorem{theorem}{Theorem}[section]
\newtheorem{lemma}{Lemma}[section]

\newtheorem{assumption}{Assumption}[section]
\newtheorem{definition}{Definition}[section]

\newtheorem{remark}{Remark}[section]
\newtheorem{example}{Example}[section]

\newreptheorem{theorem}{Theorem}
\newreptheorem{lemma}{Lemma}
\newreptheorem{corollary}{Corollary}

\newcommand{\jnote}[1]{{[\color{red}JL: #1]}}
\newcommand{\sk}[1]{{\color{magenta} SK: #1}}
\usepackage{color}
\usepackage{amssymb,amsthm}
\usepackage{amsmath}
\usepackage{enumerate}
\usepackage[shortlabels]{enumitem}
\usepackage{hyperref}
\usepackage{cite}
\usepackage{mathtools}

	\title{Provably Correct Automatic Subdifferentiation\\ for
          Qualified Programs}	
  	
	\author{  Sham Kakade
\footnote{University of Washington, Seattle, e-mail: \ sham@cs.washington.edu}
 \and Jason D. Lee
\footnote{University of Southern California, e-mail: \ jasonlee@marshall.usc.edu}
}

\begin{document}

\maketitle

\begin{abstract}
The \emph{Cheap Gradient Principle}~\citep{Griewank:2008:EDP:1455489} ---
the computational cost of computing the gradient of a scalar-valued function is nearly the same (often
within a factor of $5$)  as that of simply computing the 
function itself --- is of central importance in optimization; it
allows us to quickly obtain (high dimensional) gradients of scalar
loss functions which are subsequently used in black box gradient-based
optimization procedures. The current state of affairs is markedly
different with regards to computing subderivatives: widely used ML
libraries, including TensorFlow and PyTorch, do \emph{not} correctly
compute (generalized) subderivatives even on simple 
examples. This work considers the question: is there a \emph{Cheap
  Subgradient Principle}?  Our main result shows that, under certain
restrictions on our library of nonsmooth functions (standard in
nonlinear programming), provably correct generalized subderivatives
can be computed at a computational cost that is within a
(dimension-free) factor of $6$ of the cost of computing the scalar
function itself.   
\end{abstract}

\section{Introduction}

The widespread implementation of Automatic Differentiation (AD)
methods \citep{Baydin2015AutomaticDI} has had a transformative effect
on applied machine learning; these methods have eased the difficulty for
practitioners, across a range of disciplines, to learn sophisticated
machine learning models (including deep neural architectures
and richer inferential models).  The paradigm is: one simply writes a
program to compute the function of interest, say a scalar (loss) function
$f(x):\R^d \rightarrow \R$, and then a correctly implemented AD method
will return both $f(x)$ and all $d$ of its partial derivatives when
provided with $x$ as an input.  These partial derivatives are often
used in conjunction with some (stochastic) gradient-based optimization
approach.  

Underlying the effectiveness of this
general black-box approach
is the \emph{Cheap Gradient Principle} \citep{Griewank:2008:EDP:1455489}: the computational cost of computing the
vector of partial derivatives $(\partial f
/ \partial x_1, \partial f
/ \partial x_2, \ldots \partial f
/ \partial x_d)$ is often nearly the same as
that of simply computing the scalar function $f(x)$ itself.  In fact,
for all rational functions, the striking Baur-Strassen
theorem~\citep{Baur1983TCo,Griewank89onautomatic} shows
that this increase in computational complexity is a (dimension free) factor of $5$.

In many settings, our underlying function $f(x)$ is a nonsmooth function,
and we resort to subgradient methods.
This work considers the question:
is there a \emph{Cheap Subgradient Principle}?  Specifically, given a program that computes a (locally Lipschitz) function $f$ and given a
point $x$, can we automatically compute an element of the (Clarke)
subdifferential $\partial f(x)$ \citep{origin}, and can we do this at
a cost which is comparable to computing the function $f(x)$ itself?
Informally, the set $\partial f(x)$ is the convex hull of limits of
gradients at nearby differentiable points. It can be thought of as generalizing the gradient (for smooth functions) and the subgradient
(for convex functions).

Let us briefly consider how current approaches handle nonsmooth functions, which are available to the user as functions in some library.
Consider the following three equivalent ways to write the identity
function, where $x\in \R$,
\[
f_1(x) = x, \quad f_2(x) = \relu{x}-\relu{-x}, \quad
f_3(x) = 10 f_1(x)-9f_2(x) \, ,
\]
where $\relu{x} = \max\{x,0\}$, and so
$f_1(x)=f_2(x)=f_3(x)$. As these functions are differentiable at $0$,
the unique derivative is $f'_1(0)=f'_2(0)=f'_3(0)=1$. However, both TensorFlow~\citep{tensorflow2015-whitepaper} and
PyTorch~\citep{paszke2017automatic},
claim that $f'_1(0)=1$, $f'_2(0)=0$, $f'_3(0)=10$. This particular
answer is due to using a subgradient of $0$ at $x=0$.  One may ask if a more
judicious choice fixes such issues; unfortunately, it is not difficult to
see that no such universal choice exists\footnote{By defining
  $\textrm{ReLU}^{\prime} (0) =1/2$, the reader may note we 
  obtain the correct derivative on $f_2, f_3$; however, consider $f_4
  (x) = \relu{\relu{x}}  - \relu{-x}$, which also equals
  $f_1(x)$. Here, we would need $\textrm{ReLU}^{\prime} (0) = \frac{\sqrt{5}-1}{2}$ to
  obtain the correct answer.}.  

This example should be concerning for a number of reasons.  The use of
nonsmooth functions in AD go well beyond simple one dimensional
nonsmooth functions (such as $\relu{\cdot}$ or the
$|\cdot |$); current methods permit utilizing eigenvalues, SVDs, QR
decompositions (there are AD procedures on these nonsmooth linear
algebra functions~\citep{autograd,DBLP:journals/corr/abs-1710-08717}).

\paragraph{Is correctness important?}

One option is to disregard these
issues --- which is the current
state of affairs --- based on the observation that in most cases
these issues are unlikely to harm our optimization method. In numerical linear algebra, one could make the same argument:
we never truly encounter degenerate linear systems (or degenerate
eigenspaces); nonetheless, in retrospect, 
numerical issues have made evident the importance of carefully addressing these ``corner cases''.  The situation may be analogous here:
numerical issues in these approaches can easily lead to 
unstable outputs. Note that some numerical instability is certainly to be expected due to
nonsmoothness (a point we return to in the Discussion under the
notion of \emph{mixed stability});
yet we would still hope to have 
nontrivial stability guarantees in our widely used AD  libraries,
much in the manner we have for our established numerical linear algebra libraries~\citep{tref-nla-1997,demmel}.

Ultimately, the importance of correctness in these methods is a
decision that must be made by the broader ML community.  Here, it is
worthwhile to consider that AD software has a range of
applications: from physical simulators to health care/social science
applications to deployed online learning systems to differentiable
programming.  For example, when using physical simulators (say in robotics or in
the sciences), a strong notion of stability may be critical when doing
AD through nonsmooth system dynamics. In safety-critical settings, we
may seek to have deployed online learning methods which are not
susceptible to errors due to misspecified input-output behavior in our
programs.  Perhaps the most compelling reason for provably correct
software implementations is to avoid costly failure modes due to the
utilization of the methods in novel and unforeseen manners.


\paragraph{Related Work:}
These issues are in fact known in the mathematical AD literature (see
\citet[Chapter 14]{Griewank:2008:EDP:1455489}).  Once we include
either nonsmooth primitive functions or permit branching in a program, the usual chain rule fails to hold and incorrect input-out
behavior is easy to observe. Due to that established calculus
properties of nonsmooth functions~\citep{kk_book,Mordukhovich} do not
seem amenable to AD approaches, the current provable methods
do not have general purpose,  computationally
efficient AD methods for subdifferentials.

One influential and powerful idea is that of lexicographic
differentiation~\citep{DBLP:journals/mp/Nesterov05a}; it is a property
of a subclass of nonsmooth functions which allow these function to
inherit a generalized notion of a chain rule.  This idea has been
utilized for obtaining correct generalized derivatives in
~\citet{khan_piecewise,Griewank_linearization}. The difficulty is that
lexicographic approach often is expensive in that it involves a
dimensional factor in the computational cost increase.

The other relatively few works that do focus on automatic generalized
differentiation go through some notion of algorithmic
linearization~\citep{dijon,DBLP:journals/mp/Nesterov05a,khan_piecewise,khan_vector,griewank_study},
where often piecewise smooth functions are considered, and the
approach attempts at correct AD through probing the pieces through
some linearization (see~\citet{GRIEWANK2014} for review). The
difficulties are due to understanding what information we can extract
through linear ``probes'' into the function.  

One of the first ideas along this line of thought is due
to~\citep{dijon}, which shows how to compute directional
derivatives of nonsmooth functions through following a ``branch'' the
program would take on an input (where the branch corresponds to the   
approach direction in the directional derivative). In fact, our work
uses this basic idea, as does the ``branch locking'' approach in
\citet{doi:10.1080/10556788.2017.1341506,Griewank_linearization}. The
difficulty in these approaches is in finding a means to relate this
linearization to properties of the (nonsmooth) functions, which will allow
the algorithm to succeed; naively, we can tell when a method might
have failed though it is difficult to guarantee if it will succeed.

As such, the extant body of work does not contain methods which
contain only a constant factor blow up in the computational cost.
Notable differences in this work is that our assumptions make
strong connections to nonlinear
programming~\citep{abadie,doi:10.1137/1015075,article}, which help in characterizing when the linearization approach is informative, and we
provide a key technical result showing a certain chain rule holds for
randomized algorithms.  Furthermore, our focus is on generalizing the
reverse mode for scalar functions (as opposed to focusing on
multivariate functions where there is no known Cheap Gradient Principle).

\paragraph{Our contributions:}
Our main result provides --- under a natural set of assumptions widely
used in nonlinear programming --- a provably correct
Automatic Subdifferentiation procedure, which given some $x$, computes \emph{both} the functional value $f(x)$ and a $d$ dimensional
subdifferential $(u_1,\ldots u_d) \in \partial f(x)$, with a computational cost that is a factor of at most $6$ times that of
computing the scalar function $f(x)$ itself.  Our assumption 
is that our library of functions be implemented in a manner consistent
with the standard \emph{constraint qualification} assumptions
in nonlinear programming~\citep{abadie}.
In short, this work shows that in fact there is a
\emph{Cheap Subgradient Principle}.


\begin{algorithm}[t]
  \caption{Straight Line Program for $f(x)$} 
  \label{alg:straight_program} 

  \begin{algorithmic}[1]
    
  \item[\textbf{Input:}] $x=(x_1, \ldots x_d)$
    \FOR{$k=d+1,d+2,\ldots T$}
    \STATE Compute:
    \[
    x_k = g_{k} (x_{\text{parents}(k)})
    \]
    where 
    $\text{parents}(k)$ is the index set of the ``parent'' variables of $k$.
    \ENDFOR
    \item[\textbf{Return:}] $x_T$.
\end{algorithmic}
\end{algorithm}

\begin{algorithm}[t]
  \caption{The Reverse Mode of AD} 
  \label{alg:ReverseMode} 
    \begin{algorithmic}[1]
  \item[\textbf{Input:}] variables $(x_1, \ldots x_T)$; a
    computational graph $\{\text{children}(t)\}_{t\in \{1,\ldots T\}}$; the
    associated derivatives 
    \STATE Initialize: $
    \frac{\partial x_T}{\partial x_T} =1$
    \FOR{$t= T,T-1, \ldots 1$}
    \STATE Compute:
    \[
    \frac{\partial x_T}{\partial x_t} = \sum_{i \in \text{children}(t)} \frac{\partial x_T }{\partial x_i} \frac{\partial x_i}{\partial x_t} 
    \]
 
    \ENDFOR
    \item[\textbf{Return:}] 
$\frac{\partial x_T}{\partial x}=\left(\frac{\partial x_T}{\partial x_1},\frac{\partial x_T}{\partial x_2},\ldots \frac{\partial x_T}{\partial x_d}\right) $.

  \end{algorithmic}
\end{algorithm}

\section{Preliminaries}

Assume $f\colon\R^d\to\R$ is a locally Lipschitz
function, and recall, that by Rademacher's theorem, this implies that $f$ is
  differentiable almost everywhere. The {\em Clarke subdifferential}
of $f$ at any point $x$ is the set \citep[Theorem
8.1]{clarke2008nonsmooth} 
\begin{align}
\partial f(x):=\conv\left\{\lim_{i\to\infty} \nabla f(x_i):
  x_i\xrightarrow[]{\Omega} x\right\} \label{eq:clarke-def},
\end{align}
where $\Omega$ is any full-measure subset of $\R^d$ such that $f$ is
differentiable at each of its points.  Here, the limit is taken to be the set of all limit points.
In classical circumstances, the subdifferential reduces to more
familiar objects. Namely, when $f$ is $C^1$-smooth at $x$, the
subdifferential $\partial f(x)$ consists only of the gradient
$\nabla f(x)$, while for convex functions, it reduces to the
subdifferential in the sense of convex analysis.

\subsection{AD Review  and The Baur-Strassen Theorem}

A \emph{straight line program} for computing $f(x): \R^d \rightarrow \R$ is specified
by a program of the form shown in
Algorithm~\ref{alg:straight_program}. Here the functions $g_1,g_2,\ldots$ are
assumed to be some function from a library of functions.  In the
algebraic circuit complexity model, these functions are either
monomials or affine functions of its inputs.  

More generally, we will be interested in utilizing a richer class of
functions where $g \in \mathcal{L}$, a  library of functions, e.g. we may desire functions like
the $|\cdot|$, $\relu{x}$, or ever richer nonsmooth functions
like eigenvalues.


Define $\mathrm{Runtime}(f;x)$ to be the time it takes to compute $f(x)$ under a given
program for $f$.

\begin{theorem}
\citep{Baur1983TCo,Griewank89onautomatic} Assume all multiplications and additions have unit runtime cost.
If we restrict to the algebraic circuit
complexity model (where the functions $g_k$ are either monomials or
affine functions), then it is possible to compute
both $f(x)$ and all its partial derivatives $\nabla f(x)$ in time that is at most
$5*\mathrm{Runtime}(f;x)$.
\end{theorem}

An algorithm achieving
this guarantee is to first compute $f(x)$ and then use the \emph{reverse
mode} of AD, in Algorithm~\ref{alg:ReverseMode}.
To see the specific counting argument, see
\citep{Morgenstern1985HowTC}.  This theorem is often more general: the
reverse mode also correctly returns the derivatives even with a richer
family of smooth
functions in our library $\mathcal{L}$, often with a constant factor cost increase as
well~\citep{Griewank89onautomatic}.  The reverse mode
itself has been rediscovered many times 
\citep{griewank_invtented}; the well known back-propagation
algorithm~\citep{Rumelhart:1986:LIR:104279.104293} is one 
example of the reverse mode of AD.   The reverse mode
(and the back-propagation algorithm) is not a
direct application of the chain rule; the direct application of the
chain rule is referred to as the \emph{forward mode} of AD (see
~\citet{Griewank:2008:EDP:1455489}), which is 
$d$ times more expensive procedure to compute the gradient.  The reverse mode can
be viewed as a form of dynamic programming. To compare the two, in
the reverse mode of AD, we compute the derivatives $\frac{\partial
  x_T}{\partial x_t}$, referred to as the adjoints\footnote{For a variable $x_T = g(x_{\text{parents}} )$, the notation $\frac{\partial x_{T}}{\partial x_t}$ refers to the derivative with respect to $x_t$, but holding all parent variables of $x_t$ as fixed. If $x_t$ is an input variable, then this is the usual partial derivative.}, while in the
forward mode of AD we would compute ($d$-dimensional) derivatives of the form $\frac{\partial
  x_t}{\partial x}$ (referred to as dual numbers).

\subsection{Nonsmooth functions and our computational model}

To specify how our nonsmooth functions are implemented, we
extend the computational model to allow for branching, using (a
restricted version\footnote{We avoid halting
  concerns by assuming our programs halt in a bounded amount of time. We also explicitly avoid discussing tapes and registers in
  our computational cost model.} of) the Blum-Shub-Smale model of
computation~\citep{DBLP:conf/focs/BlumSS88}.

\begin{definition}[Computation Model]
The computational model for computing any
$g(x): \R^d \rightarrow \R$ in our library ($d$
may be different for each function) is specified 
by a program of the form shown in Algorithm~\ref{alg:branch_program}. We
assume that the function $g_{k,z}$ is either a monomial or an affine function of its inputs.  
Furthermore, for every $g$, we assume that there exists
a time $T$, where the program terminates in at most this amount
of time.
\end{definition}

\begin{algorithm}[t]
  \caption{Program for a Nonsmooth function $g(x)$ } 
  \label{alg:branch_program} 

  \begin{algorithmic}[1]

  \item[\textbf{Input:}] $x=(x_1, \ldots x_d)$
    \STATE Initialize a vector $z$ to be all $-1$'s. $z$ is for
    notational convenience to keep track of the branch.
    \FOR{$k=d+1,d+2,\ldots T$}
    \STATE Compute:
    \[
    x_k = g_{k,z} (x_{\text{parents}(k,z)})
    \]
    \STATE If the program branches at $(k,z)$, then\\
    \begin{itemize}
    \item If: $x_k \ge 0$, $z_k=1$.
    \item  Else: $z_k=-1$.
    \end{itemize}
    \STATE  If the program halts at $(k,z)$, then terminate the \textbf{for} loop. 
    \ENDFOR
    \item[\textbf{Return:}] $x_k$.
\end{algorithmic}
\end{algorithm}

Throughout, we make the following assumption:

\begin{assumption}\label{assumption:cost}
(Computational Cost) Assume all multiplications and additions have unit runtime cost
  and that an execution of an ``If''
  statement is also unit cost.  For example, the cost of computing a
  monomial is the number of multiplications.
\end{assumption}

The program implicitly encodes a
function that has the following representation:
\begin{align}
f(x) = \sum_{z \in \{-1,1\}^T}  \mathbb{I}_{S_z} (x) p_z (x),
\label{eq:piecewise-def-f}
\end{align}
where each $p_z$ is a polynomial; $\mathbb{I}_{S_z}$ is 
the indicator function on the set $S_z$; and $S_z$ consists of
all $x$ where the program executes branch $z$ when given $x$ as
input. The set $S_z$ can be explicitly defined as follows: for 
steps $k$ where the programs branches on $z$, define
$h_{k,z}(x)=x_k$; on non-branching $k$, define $h_{k,z}(x)=-1$; define
the vector valued function $h_z(x)=(h_1(z),\ldots h_T(x))$; 
\begin{equation} \label{eq:Sz-def}
S_z := \{ x | \, \mathbb{I} (\sign(h_z(x))=z) \}
\end{equation}
where the $\sign(\cdot)$ is the usual sign function (applied componentwise) taking values in
$\{-1,1\}$ (where we take $\sign(0)=1$). Note that $S_z$ is
specified by a set of polynomial inequalities as defined by the
functions  $h_{k,z}(x)$.

\section{Provable Automatic Subdifferentiation}

\begin{figure}[t]

\begin{minipage}{.45\textwidth}
\begin{algorithm}[H]
  \caption{$\relu{x}$} 

  \begin{algorithmic}[1]
    
  \item[\textbf{Input:}] $x=x_1$
    \STATE Branch:
    \begin{itemize}
    \item If: $x_1 \ge 0$, set $x_2=x_1$.
    \item  Else:  set $x_2=0$.
    \end{itemize}
    \item[\textbf{Return:}] $x_2$.
\end{algorithmic}
\end{algorithm}
\end{minipage}
\begin{minipage}{.45\textwidth}
\begin{algorithm}[H]
  \caption{$\relu{x}$} 

  \begin{algorithmic}[1]
    
  \item[\textbf{Input:}] $x=x_1$
    \STATE Branch:
    \begin{itemize}
    \item If: $x_1^3 \ge 0$, set $x_2=x_1$.
    \item  Else:  set $x_2=0$.
    \end{itemize}
    \item[\textbf{Return:}] $x_2$.
\end{algorithmic}
\end{algorithm}
\end{minipage}

\caption{\label{fig:relu} Two programs that implement $\relu{x}$:
  Both programs are correct and return
  the same value. However, the program on the right violates
  Assumption \ref{assumption:cq} since the gradient of the constraint
  function at $x=0$, $\nabla (x_1 ^3) = 3x_1^2 =0$.  }
\end{figure}

In the algebraic circuit complexity model, where AD is provably
correct,  branching is not permitted. The inclusion of branching into
our program leads to a number of subtle issues.
Branching allows us to implement the same nonsmooth
function in different manners, which have important
consequences in linearization approaches. Consider two different
programs (with the same input-output behavior) for the $\relu{x}$
function in Figure~\ref{fig:relu}. The left program returns $x$ on the
constraint set that is encoded as $S_1=\{x| x\geq 0\}$, while the right
program returns $x$ on the constraint set that is encoded as $S_1=\{x|
x^3\geq 0\}$.  In nonlinear programming, the importance of avoiding
encoding constraints in the latter manner is well-known~\citep{abadie,doi:10.1137/1015075,article}.

This example motivates our 
restriction to only consider library
functions that are encoded like the former set.  
We will make the 
standard constraint
qualification assumption\footnote{The 
standard constraint
qualification assumption on a constraint set is that the tangent
cone of the constraint set equals the linearized cone (of the
functions which define the constraints).}. Roughly
speaking, the assumption states that first order information 
characterizes the set of feasible perturbations. We state this
assumption in a manner more directly applicable to our setting (see
\citep{abadie,doi:10.1137/1015075,article}).

\begin{assumption}
\label{assumption:cq}
(Constraint Qualification on our Library)
Assume for all $g\in
\mathcal{L}$ that $g$ is locally Lipschitz and our program for $g$ (in our computational model) satisfies the constraint
qualification condition on all sets $S_z$ in the following sense:
suppose $\{h_z\}$
(for binary $z$)  are the
corresponding constraint functions in our program. For
any $x,v$ (of the same input dimensionality of $g$), assume that for
all $z$:
\[
\lim_{\delta \downarrow 0} (\sign (h_z(x+\delta v)))
=
\lim_{\delta \downarrow 0} 
(\sign (h_z(x)+\delta\nabla h_z (x)\cdot v))
\, .
\]
Roughly, this states that the set approached along the limiting direction $x + \delta
v$, when $\delta \downarrow 0$, can be determined with first order information.
\end{assumption}


Before we state our main theorem, one more definition is in order, due to that
$\mathrm{Runtime}(f;x)$ may not be continuous. Define the
limiting runtime $\mathrm{Runtime}^*(f;x)$ of $f$ at $x$ as the (supremum) runtime
to compute $f(x)$, as $x$ is approached
from nearby points. Precisely,
\[
\mathrm{Runtime}^*(f;x):=\sup \left\{\lim_{i\to\infty} \mathrm{Runtime}(f;x_i):
  x_i\rightarrow x\right\},
\]
(where the limit is taken to be the set of all limit points).


\begin{theorem}\label{thm:main}
(A Cheap Subgradient
  Principle) Assume that our
program for
$f(x)$, in Algorithm~\ref{alg:straight_program}, is allowed to use nonsmooth functions from our library
$\mathcal{L}$ (in addition to affine functions and monomials).
Suppose assumptions~\ref{assumption:cost} and ~\ref{assumption:cq} hold. There
exists a (randomized) algorithm, which upon input $x$,
terminates in time that is at most
$6*\mathrm{Runtime}^*(f;x)$, and, almost surely, returns both $f(x)$
and an element $u \in \partial f(x)$.
\end{theorem}

The following example shows one subtle issue with regards to
constraint qualification.

\begin{example}\label{example:relu_x_squared}
(Constraint qualification on programs do not compose)  Consider the function
$f(x)=\relu{x^2}$ (which is equivalent to the smooth
  function $x^2$). It is straight forward to see that the induced program for
  $f(x)=\relu{x^2}$ (when we unravel it) does not satisfy the
  constraint qualification assumption, even if we do use an
  implementation of $\relu{\cdot}$ that does satisfy this
  assumption. Regardless, in Example~\ref{example:return}, we show
  that our algorithm does indeed correctly compute the gradient on
  this (continuous) function.
\end{example}

Before we present the construction, we first provide a chain rule for
nonsmooth functions.

\subsection{A Chain Rule for Nonsmooth Functions}

Let $D[g;v](x)$ denote the one-sided (Dini) directional derivative:
\[
D[g;v](x) := \lim_{\delta \downarrow 0} \frac{g(x+ \delta v) -g(x)}{\delta}.
\]
(note that we are not assuming that $v$ is a unit vector). This
derivative exists for all piecewise polynomials and semialgebraic functions~\citep[Lemma 6.2]{coste2000introduction}.

\begin{assumption}
\label{assumption:library}
(Overloading the library with ASD subroutines)
Assume we have a library of (locally Lipschitz) functions
$\mathcal{L}$ computable in our computational model.
For any $g \in \mathcal{L}$, with the
representation $g(x)=\sum_{z \in \{-1,1\}^T}  \mathbb{I}_{S_z} (x) p_z (x)$,
assume we have the following associated automatic
subdifferentiation subroutine $\text{ASD}[g]$ with the
following behavior: upon input $(x;v)$, the output $[a, d, u] =
\text{ASD}[g](x;v)$ satisfies
\[
a =g(x), \, 
d = D[g;v](x), \,
u = \nabla p_z(x)
\]
where $z$ is such that:
\[
\lim_{\delta \downarrow 0} (\mathbb{I}_{S_z}(x+\delta v))=1\, .
\] 
Roughly speaking, $u$ is the derivative determined by the set $S_z$ which is
approached along the limiting direction $x + \delta v$, when $\delta \downarrow 0$.
\end{assumption}

For any locally Lipschitz function $h$, define the limiting total derivate as:
\[
\partial[h;v](x) := \lim_{\delta \downarrow 0} \nabla h(x+ \delta v)
\]
if the limit exists. For almost all $v$,  the limit exists, and $\partial[h;v](x)$ is a
subdifferential of $h$.


\begin{theorem}\label{thm:chain}
(A Chain Rule for Nonsmooth Functions) 
Assume $h:\R^m \rightarrow \R$ and $g_1$, \ldots
$g_m$ (where $g_i: \R^d \rightarrow \R$)  are locally Lipschitz functions computable in our
computational model and that the function $h$ is overloaded with an ASD subroutine
as specified in Assumption~\ref{assumption:library}.
Define:
\[
f(x) := h(g_1(x), \ldots g_m(x)) =h(g(x))\, ,
\]
where $g(x)$ is the vector valued function $(g_1(x), \ldots
g_m(x))^\top$. 
Denote the $m\times 1$ vector of (one-sided) directional
derivatives as $D[g;v](x)$. If it exists, let $\partial[g;v](x)$ denote $m
\times d$ limiting Jacobian
matrix (whose rows are given by the vectors $\partial[g_i;v](x)$'s). Set:
\[
[a, d, u] = \text{ASD}[h] (g(x); D[g;v](x))
\]
For all but a measure $0$ set of $v$, we have that $\partial[f;v](x)$
and $\partial[g;v](x)$ exist and that:
\begin{align}
\partial[f;v](x) = \partial[g;v](x)^\top u \, .
\label{eq:chain-rule}
\end{align}
\end{theorem}

\begin{example}
Consider the example $x= f_2 (x) = \relu{x} - \relu{-x}$. We define $h
(y_1,y_2) =y_1- y_2$, $g_1 (x) = \relu{x}$, and $g_2(x) = \relu{-x}$,
so that $f_2 = h(g_1(x),g_2 (x))$. By applying the \text{ASD} subroutine
to $h$, starting at $x=0$ with $v=1$ which leads to running
$\text{ASD}[h] ((0,0);(1,0))=[a,d,u]$ (where it is straightforward to
verify that $u=[1, -1]^\top$), we obtain 
\begin{align*}
\partial[f_2;v] (0) &= \partial[g; v](0) ^T u\\
&= \begin{bmatrix} 
1 \\ 0 
\end{bmatrix} ^\top
\begin{bmatrix} 1 \\ -1 \end{bmatrix}\\
&=1,
\end{align*}
which is correct. Furthermore, note a correct answer is obtained for
any $v \neq 0$.
\end{example}


\begin{example} 
  We return to $f(x)=\relu{x^2}$ from
  Example~\ref{example:relu_x_squared}. Define $h(y)=\relu{y}$, $g(x) = x^2$, and so
  $f(x) = h(g(x))$. By applying the chain rule lemma at $x=0$ with $v=1$,
\[
\partial[f;v] (0) = \partial[g; v](0)  u = 0\cdot u = 0
\]
Subtly, note that $[a, d, u] = \text{ASD}[h] (0; 0)$, so we are
feeding a degenerate direction $d=0$ into our subroutine. Regardless, the
chain rule lemma still applies (for any $v$ in this case). 
\end{example}

\subsection{The algorithm}

We first present the algorithm that utilizes an overloaded library. We then provide a provably correct
construction of this overloaded library. All proofs are
provided in the appendix. 

\subsubsection*{Subdifferentiation with the overloaded library}

Algorithm~\ref{alg:asd-with-library} is the Automatic Subdifferentiation
procedure. Correctness follows from Lemma \ref{lemma:library}.

\begin{algorithm}[t]
  \caption{Automatic Subdifferentiation} 
  \label{alg:asd-with-library}   
  \begin{algorithmic}[1]
  \item[\textbf{Input:}] $x=(x_1, \ldots x_d)$, $v\in R^d$. \\
  \item[\textbf{Initialize:}]
    Set    $\dot{x}_1=v_1,\, \dot{x}_2=v_2, \ldots  \dot{x}_d=v_d$.
    \FOR{$k=d+1,d+2,\ldots T$}
    \STATE Compute $[a, d, u] = \text{ASD}[g_k]
    (x_{\text{parents}(k)};\dot{x}_{\text{parents}(k)})$ and set: \label{state:AD-state}
    \[
    x_k = a, \, 
    \dot{x}_k = d, \quad
    \frac{\partial x_k}{\partial x_{\text{parents}(k)}} = u
    \]
    \ENDFOR
    \STATE Compute $\frac{\partial x_T}{\partial x}$ using the Reverse Mode on these precomputed variables.
    \item[\textbf{Return:}] $x_T$, 
and $\frac{\partial x_T}{\partial x}$.
  \end{algorithmic}
\end{algorithm}

\begin{algorithm}[t]
  \caption{Overloading the function $g(x)$ } 
  \label{alg:ASD} 
  
  \begin{algorithmic}[1]
  \item[\textbf{Input:}] $x=(x_1, \ldots x_d)$, $v\in R^d$. \\
  \item[\textbf{Initialize:}] Set    $\dot{x}_1=v_1,\, \dot{x}_2=v_2, \ldots  \dot{x}_d=v_d$.
    \FOR{$k=d+1,d+2,\ldots T$}
    \STATE Compute $x_k$, its partial derivatives, and the directional derivative:
    \begin{align*}
    x_k &= g_{k,z} (x_{\text{parents}(k,z)})\, , \quad
    \left\{ \frac{\partial x_k}{\partial x_j} \middle \vert j \in \text{parents}(k,z) \right\}, \,
    \\
    \dot{x}_k &= \sum_{j \in \text{parents}(k,z)} \frac{\partial x_k}{\partial x_j} \dot{x}_j 
    \end{align*}
    \STATE If the program branches at $(k,z)$, then:
    \begin{itemize}
    \item If: $x_k > 0$, then $z_k=1$.
    \item Elseif: $x_k = 0$ \emph{and} $\dot{x}_k \geq 0$, then $z_k=1$.
    \item Else: $z_k=-1$
    \end{itemize}
    \STATE  If the program halts at $(k,z)$, then terminate the \textbf{for} loop. 
    \ENDFOR
    \STATE Compute $\frac{\partial x_k}{\partial x}$ using the Reverse Mode on these pre-computed variables.
    \item[\textbf{Return:}] $[a,d,u] = [x_k, \dot{x_k},\frac{\partial x_k}{\partial x}]$.
  \end{algorithmic}
\label{alg:asd-subroutine}
\end{algorithm}

\begin{lemma}\label{lemma:library}
Suppose
Assumptions \ref{assumption:cost} and \ref{assumption:library} hold. 
Upon input of an arbitrary $x$, and if $v$ is sampled uniformly at random from
the unit sphere, then, almost surely, Algorithm~\ref{alg:asd-with-library}
returns both $f(x)$ and an element $u \in \partial f(x)$.
\end{lemma}
\begin{proof}
  Fix $k \in[d+1,\ldots, T]$. Every parent variable
  $j \in \text{parent}(k)$ can be expressed as
  $x_j = \tilde{g}_j (x) $, where $g_j$ is a piecewise polynomial on
  the $d$ dimensional input $x$. Thus
  \begin{align*}
    x_k = g_k ( \tilde{g}_1 (x),\ldots, \tilde{g}_{k-1} (x)).
  \end{align*}
  Now the usual chain rule holds for directional
  derivatives~\citep{Shapiro1990}. As the forward mode of AD implements
  the usual chain rule of directional derivatives, then we have
  $\dot{x}_{j}= D[\tilde g_j;v]$.
	 
	  By Assumption \ref{assumption:library} and Theorem \ref{thm:chain}, $\text{ASD}[g_k](x_{\text{parents}(k)}, \dot{x}_{\text{parents}(k)})$ returns $u= \frac{\partial x_k}{\partial x_{\text{parents}(k)}}= \partial [g_k ; \dot{x}_{\text{parents}(k)}]$ and this limiting total derivate satisfies the chain rule $\partial[x_k; v] (x) = \partial[\tilde g;v](x) ^\top u$. Since the limiting total derivates satisfies the chain rule and the validity of reverse mode AD algorithm relies only on the chain rule, Algorithm \ref{alg:asd-with-library}  correctly computes $\partial [f(x); v]$.
	  
By Rademacher's theorem and the definition of Clarke subgradient in Equation \eqref{eq:clarke-def}, $\partial [f(x);v] \in \partial f(x)$, for almost all $v$.
\end{proof}	 
	 
	
	

\subsubsection*{Overloading the Library Functions}

The following lemma shows that we can provide a method to correctly
overload the library, which we use in Algorithm~\ref{alg:asd-with-library}.

\begin{lemma}\label{lemma:black_box}
(Correct Library Overloading) Assume $g$ 
satisfies the constraint qualification
  conditions in Assumption~\ref{assumption:cq},
Suppose the corresponding representation is $g(x)=\sum_{z \in \{-1,1\}^T}  \mathbb{I}_{S_z} (x) p_z (x)$,
On an arbitrary input $x$ and $v$, Algorithm~\ref{alg:ASD} returns $g(x)$,  $D[g;v](x)$,
  and an element $u =\nabla p_z(x)$
where $z$ is such that:
$
\lim_{\delta \downarrow 0} (\mathbb{I}_{S_z}(x+\delta v))=1\, .
$
\end{lemma}

\begin{example}\label{example:return}
We again return to $\relu{x^2}$ from Example~\ref{example:relu_x_squared}. Here
we examine how  $h$ is overloaded based on the
 implementation in Algorithm \ref{alg:ASD}. When $(x,v)=(0,1)$, we are
 running $\text{ASD}[h] (0; 0)$ and this may not follow the same
 branch had we run on the (infinitesimal) input $x = \epsilon
        v$ which leads to running $h(\epsilon^2 v^2)$. However, the gradient is correctly computed, $\partial \relu{x^2} = 0$,
	regardless of the branch taken.	
\end{example}

\section{Discussion and Open Questions}


{\bf Overloading the Library Functions:} It is not difficult to see that piecewise univariate functions can be
implemented in our library. 

\begin{example}{Univariate Piecewise Polynomial (Algorithm~\ref{alg:univariate-poly}).}
	Let $\sigma: \reals \to \reals$ be a univariate piecewise polynomial, meaning that the domain $\reals$ is partitioned into a set of $k$ intervals  $(-\infty, b_1),(b_1,b_2),\ldots, (b_{k-1},\infty)$. On each interval, the function is equal to a polynomial $p_1,\ldots, p_k$.
	
	Algorithm \ref{alg:univariate-poly} provides a constraint qualified program for the function $\sigma(\cdot)$, which can be used as a library function.
	\end{example}
\begin{figure}[t]

		\begin{algorithm}[H]
				\label{alg:univariate-poly}
			\caption{$\sigma(x)$} 
			
			\begin{algorithmic}[1]
				
				\item[\textbf{Input:}] $x=x_1$
				\STATE Branch:
				\begin{itemize}
					\item If: $x_1 \le b_1$, set $x_2=p_1 (x)$.
					\item  Elseif: $x_1 \le b_2$, set $x_2 =p_2 (x)$.\\
				~	\vdots\\
					\item Elseif: $x_1 \le b_{k-1}$, set $x_2 =p_{k-1} (x)$.\\
					\item Else: set $x_2 = p_k (x)$.
				\end{itemize}
				\item[\textbf{Return:}] $x_2$.
			\end{algorithmic}
		\end{algorithm}

\end{figure}

An important step would be in extending our computational model to
allow the incorporation of provably correct automatic
subdifferentiation libraries for linear
algebra libraries.  AutoGrad~\citep{autograd} does do AD through linear
algebra methods though it can not be used to obtain correct
subdifferentials in programs (at
nondifferentiable points);
obtaining correct generalized derivatives may be particularly important in cases where we deal with low
rank methods. We conjecture our results can be extended, by extending
the computational model, to handle these cases
(there is already much known about the first order structure of
these methods~\citep{DBLP:journals/corr/abs-1710-08717}); technically,
SVDs are not exactly computable in either the algebraic circuit
complexity model or the Blum-Shub-Smale model.

{\bf Numerical Analysis:} The most important open question is how to
obtain numerically stable and accurate
solutions~\citep{tref-nla-1997,demmel}.  We conjecture the techniques
developed here will help in characterizing these issues. In
particular, the most natural question is how to develop algorithms that
satisfy the \emph{mixed stability} criterion: the algorithm should give
``nearly the right answer to nearly the right problem'' (as in
~\citep{tref-nla-1997}). For example, for the $\text{abs}(\cdot)$ function, it
should be acceptable for an AD method to provide a subgradient near to
$-1$ for a small input $\epsilon>0$ due to roundoff error; however, it would undesirable for 
numerical error to lead vastly different gradients than those that
arise from any nearby
problem. This may be particularly important when doing AD in
physical simulators. 


\paragraph{Acknowledgments:}  We thank Dima Drusvyatskiy for many
helpful discussions. Sham Kakade acknowledges funding from
Washington Research Foundation Fund for Innovation in Data-Intensive
Discovery, the NSF through award CCF-1740551, and ONR award N00014-18-1-2247.  Jason D. Lee acknowledges support of the ARO under MURI Award W911NF-11-1-0303.  This is part of the collaboration between US DOD, UK MOD and UK Engineering and Physical Research Council (EPSRC) under the Multidisciplinary University Research Initiative.

\bibliography{bibliography,ad_bib}
\bibliographystyle{plainnat}

\appendix
\section{Proofs}
\label{sec:proof}

First, we provide a helpful technical lemma.

\begin{lemma}\label{lem:piecewise-direction}
Let $z \in \{-1,1\}^T$,  and suppose the functions $p_1(x), \ldots p_T(x)$ are analytic
functions from $\reals^n\rightarrow \reals$.
  Let $S_z \subset
\reals^n$ be a set defined as follows:
$$
S_z=\{x \, \vert \,  \sign(p_1(x))=z_1, \ldots \sign(p_T(x))=z_T \} \, .
$$
Fix any $x\in \reals^n$. For all $z \in \{-1,1\}^T$ and for almost all $u$, there
is an open neighborhood $ B(u,r)$ such that for all $u' \in B(u,r)$ ,
\[
\lim_{\delta \downarrow 0} \mathbb{I}_{S_z} (x+\delta u)=\lim_{\delta \downarrow 0} \mathbb{I}_{S_z} (x+\delta u') \, .
\]
\end{lemma}

\begin{proof}
We need only consider one such $z$ for the proof, as the set of
possible $z$ is finite. Without loss of generality, we can say $S_z=S $ is a set of the form:
	$$
	S=\{p_1(x)\ge 0,\ldots, p_k(x) \ge 0, p_{k+1}(x) <0,\ldots,p_T(x) <0\} \, .
	$$
For the proof, it suffices to show that for all $j$ and almost all $u$, there
is an open neighborhood $ B(u,r)$ such that for all $u' \in B(u,r)$ ,
\[
\lim_{\delta \downarrow 0} \sign(p_j(x+\delta u)) =
\lim_{\delta \downarrow 0} \sign(p_j(x+\delta u'))
\, .
\]
Let us split the constraints into a set of active constraints $p_j(x)
=0$ for $j \in A$ and inactive constraints $p_j (x) \neq 0$ for $j \in I$.
For inactive constraints, the above holds due to continuity. It
remains to show the above for active constraints. Let us also assume that
the functions $p_j$ in the active set are not identically equal to $0$ (the claim
is true for any zero function).
	

For every active constraint $j \in A$ , testing $\lim_{\delta \downarrow 0}
\sign(p_j(x+\delta u) )= 1$ can be done by the Taylor expansion,
\begin{align*}
  p_j(x+\delta u) = p_j(x)+ \sum_{i=1}^\infty\frac{\delta^i \norm{u}^i}{i!}
  \nabla^i p_j(x) [u^{\otimes i}] \, , 
\end{align*}
which exists since $p_j$ is analytic.
For any $i$, note that $\nabla^i p_j(x) [u^{\otimes i}]$ is a
polynomial function in $u$. As it is a polynomial, this function will
either be equal to $0$ (for all $u$) or
it will be nonzero for almost all $u$\footnote{This can proven
  by induction on the dimension of $u$. }.
Let $i_j$ be the first index that the function
$\nabla^i p_j(x) [u^{\otimes i}]$ is nonzero for almost all $u$ ($i_j$
is finite as $p_j$ is not the zero function). Then for almost every direction $u$,
\[
  \lim_{\delta \downarrow 0} \sign (p_j (x+\delta u) ) =\sign\big(
  \nabla^{i_j} p_j(x) [ u^{\otimes i_j} ]\big). 
\]
For almost all $u$,
$\nabla ^{i_j} p_j(x) [u^{\otimes i_j}] \neq 0$. Since $\nabla ^{i_j}
p_j(x) [u^{\otimes i_j}]$ is strictly non-zero, then, 
by continuity,
\[
\sign (\nabla ^{i_j} p_j(x) [u^{\otimes i_j}])
=
\sign(\nabla^{i_j} p_j (x) [ u'^{\otimes i_j}])
\]
for $\norm{u'-u}<r$ (for sufficiently small $r$). This implies that for $\norm{u'-u}<r$
\[
\lim_{\delta \downarrow 0} \sign(p_j(x+\delta u)) =
\lim_{\delta \downarrow 0} \sign(p_j(x+\delta u'))
\, ,
\]
which completes the proof.
\end{proof}

\subsection{Proof of Theorem~\ref{thm:chain}}

\begin{proof}
Suppose that a program for $h(y)=h(y_1,\ldots, y_m)$ has the representation:
\begin{eqnarray*}\label{eq:h}
h(y) = \sum_{z \in \{-1,1\}^T}  \mathbb{I}_{S_z} (y) p_z (y),
\end{eqnarray*}
where $S_z$ is specified as in Equation~\ref{eq:piecewise-def-f}.
Since $f$ is the composition of two programs, it also has a
representation of the form:
\begin{eqnarray*}
f(x) = \sum_{z \in \{-1,1\}^{T'}}  \mathbb{I}_{\tilde{S}_z} (x) \tilde{p}_z (x),
\end{eqnarray*}
for different polynomials $\tilde{p}_z$ and sets $\tilde{S}_z$.

By Rademacher's theorem, $\partial [f;v]$ exists for almost all
$v$. Since $f$ itself is a piecewise polynomial, $\partial[f;v](x) =
\nabla \tilde{p}_z(x)$ where $z$ is such that $\lim_{\delta \downarrow 0}
\mathbb{I}_{\tilde{S}_z} (x+\delta v)=1$.
Using Lemma \ref{lem:piecewise-direction}, for all $z' \in
\{-1,1\}^{T'}$ and almost all $v$, 
there is a (full dimensional) neighborhood $\mathcal{B}$ around $v$,
in which $\lim_{\delta \downarrow 0} \mathbb{I}_{\tilde{S}_{z'}}
(x+\delta v')=1$, which implies $\partial[f;v'](x) = \nabla
\tilde{p}_{z}(x)$ (using our choice of $z$). Hence,  
\[
\partial[f;v](x) = \partial[f;v'](x) \, 
\]
for all $v' \in \mathcal{B}$.

Furthermore, using this and the definition of the directional
derivative, we have for all $v' \in \mathcal{B}$,
\begin{equation}\label{eq:dir_deriv_pbar}
D[f;v'](x)= \partial[f;v'](x) ^\top v' =\partial[f;v](x)^\top v'
\, .
\end{equation}
The remainder of the  proof seeks to show this directional derivative can be
written as: 
\[
D[f;v'](x)= u^\top \partial[g;v](x) \, v'
\]
(for $v'$ in a sufficiently
small full dimensional ball). Note this would imply that for all $v'$
in a full dimensional ball,  $u^\top \partial[g;v](x) \, v'
=\partial[f;v](x)^\top v' $, which would complete the proof by choosing $d$
linearly independent $v'$'s. 

Define $M=\partial[g;v](x)\in \mathbb{R}^{m \times d}$ be a matrix whose rows are $\partial[g_i;v](x)$.   Observe $D[g;v](x) =M v$. Now let us
show that:
\[
D[g;v'](x) =M v'
\]
for all $v' \in \mathcal{B}_1$, where $\mathcal{B}_1 $ is a
sufficiently small (full dimensional) neighborhood
around $v$.
To see this, note the function $g(x)$ itself is a piecewise (vector) valued
function, and, by Lemma \ref{lem:piecewise-direction}, for almost all
$v$, there is a neighborhood
$\mathcal{B}_1$ around $v$, in which for all $v' \in
\mathcal{B}_1$, $v'$ selects the same piece as $v$, which implies 
$\partial[g;v'](x)=M$ as claimed.

Now due to the assumed properties of $ASD$, we know
\[
u = \nabla p_{z}(y) \vert_{y=g(x)} \, .
\]
for some $z$ which satisfies:
\[
\lim_{\delta \downarrow 0} \mathbb{I}_{S_z} (g(x)+
\delta D[g;v](x) ) = 1 \, .
\]
Since $D[g;v](x) =M v$, we have:
\[
\lim_{\delta \downarrow 0} \mathbb{I}_{S_z} (g(x)+
\delta Mv ) = 1 \, .
\]
Also, note
\[
D[h;Mv](y) =\nabla  p_z(y)^\top Mv 
\]
by the definition of the directional derivative and where
$y=g(x)$.

Using Lemma
\ref{lem:piecewise-direction} again, we will show that
\begin{equation}  \label{eq:claim}
D[h;Mv'](y) =\nabla  p_z(y)^\top Mv' 
\end{equation}
for all $v' \in \mathcal{B}_2$ (for a sufficiently small ball
$\mathcal{B}_2$). Note that the above holds if:
\begin{align}
\lim_{\delta \downarrow 0} \mathbb{I}_{S_z} (g(x)+
\delta Mv' ) = 1 \, , \label{eq:Mv'-set}
\end{align}
for all $v' \in \mathcal{B}_2$.
Now consider the set
$\bar{S}_z$ defined by
\[
\bar{S}_z=\{\tilde v \, \vert \,  \sign(p_1(g(x)+M\tilde
v))=z_1, \ldots \sign(p_T(g(x)+M\tilde v))=z_T \}.
\]
Lemma \ref{lem:piecewise-direction} guarantees
that  there exists a  ball $\mathcal{B}_2$ (centered
on $v$) such that:
\[
\lim_{\delta \downarrow 0} \mathbb{I}_{\bar{S}_{z}} (
\delta v' ) = 1 \, ,
\]
for all $v' \in \mathcal{B}_2$, which is equivalent to Equation
\eqref{eq:Mv'-set}. Thus we have proven the Equation \eqref{eq:claim},
for all $v' \in \mathcal{B}_2$.

As the chain rule holds for directional derivatives (of locally Lipschitz piecewise polynomial/semialgebraic  functions)~\citep{Shapiro1990}, we have that:
\[
D[f;v'](x) =D[h;Mv'](y) =\nabla  p_z(y)^\top Mv' =u^\top \partial[g;v](x) v'
\]
for $v'$ that is in some sufficiently small full dimensional neighborhood
$\mathcal{B}_3$ (that is contained in $\mathcal{B}_2$ and $\mathcal{B}_1$).

Now define the full dimensional neighborhood $\mathcal{B}_4$ so that it is contained in
both $\mathcal{B}_3$ and $\mathcal{B}$. We have shown, for $v' \in \mathcal{B}_4$,
\[
u^\top \partial[g;v](x) \, v' =\partial[f;v](x)^\top  v' \, ,
\]
using Equation~\ref{eq:dir_deriv_pbar}. The proof is completed by choosing $d$ linearly independent $v'$'s, which implies that $u^\top \partial[g;v](x) =\partial[f;v](x)^\top$.
\end{proof}


\subsection{Proof of Lemma~\ref{lemma:black_box}}

\begin{proof}[Proof of Lemma~\ref{lemma:black_box}]
Let $h_{t,z} (x) $ be polynomials that represent  $g_{t,z} $ by expressing in terms of the input variables $x$ only. By the Constraint Qualification Assumption \ref{assumption:cq}, the branch taken by Algorithm \ref{alg:asd-subroutine} is the same branch that would be taken by running on the perturbed input $x+\delta v$ for all $\delta< \delta_0 (v)$.

To finish the proof, we simply note that the reverse mode is performed on a straight-line program that computes $\nabla p_z (x) $, where $z$ is such that $\lim_{\delta \downarrow 0} \mathbb{I}_{S_z} (x+\delta v) =1$. 
\end{proof}

\subsection{Completing the proof of Theorem~\ref{thm:main}}
Lemma \ref{lemma:black_box} shows that every library function $g$ that
satisfies Assumption \ref{assumption:cq} can be implemented via
Algorithm \ref{alg:asd-subroutine} to satisfy Assumption
\ref{assumption:library}. Then Lemma \ref{lemma:library} shows that we compute an element $\partial[f;v](x) \in \partial f(x)$.

The second part of the proof counts the number of unit operations. See~\citet{Morgenstern1985HowTC,Griewank89onautomatic} for the counting
argument for a factor of $5$, where the reverse
mode (after we have computed the function and stored the intermediate
variables) is a factor of $4$ more than computing the function
itself. 

To obtain a factor of $6$ in the existence claim in
Theorem~\ref{thm:main}, the construction uses a minor modification of
the 
algorithm presented, where the algorithm is modified to only run one (global)
reverse mode (as opposed to using a reverse mode computation
associated with each library function). Algorithm~\ref{alg:asd-with-library},
the one presented, has a runtime that is within a factor of $10$ of the cost of
computing just $f(x)$, due to that it runs the reverse mode when
calling every library function (which costs a factor of $4$ more than just
computing $g(x)$). This algorithm presented is what we actually suggest
to use, as the algorithm that achieves a factor of
$6$ may require more memory in order to build a larger computational
graph (and the factor of $10$ is rather pessimistic). To see the factor of $10$, the reverse mode 
(associated with each library function) gives an additional factor of $4$.
The directional derivative
costs one additional factor (this is referred to as the forward mode of AD~\citep{Griewank89onautomatic}). After all $T$ (overloaded) library
functions are completed, one additional reverse mode is called at the
end of Algorithm~\ref{alg:asd-with-library}, which incurs (at most) an additional factor of $4$.  

To obtain the factor of $6$, it is straightforward to modify the
algorithm, as the following proof shows.

\begin{proof}[Proof of Theorem~\ref{thm:main}]
First, observe that the algorithm correctly computes the function
value on some (limiting branch), so the time it takes to compute
$f(x)$ on this branch is some $T$ where $T \leq
\mathrm{Runtime}^*(f;x)$. Now the directional derivative computation can
be computed when doing the forward pass of the algorithm, which incurs
at most one additional factor of $T$ (this is also referred to as the
forward mode of AD).  Instead of computing the intermediate
derivatives $\frac{\partial x_k}{\partial x_j}$ (as in
Algorithm~\ref{alg:ASD}), we will unravel the entire computational
graph of $f(x)$, where we replace each function $g_k(\cdot)$, by
inserting the corresponding computational
graph of $g_k(\cdot)$ into a global computational graph. We then run the reverse mode
of AD on this global computational graph, giving only one additional factor
of $4$.
\end{proof}

\end{document}